\documentclass[12pt,reqno]{amsart}

\usepackage{amsmath, amsthm, amsopn, amssymb}
\usepackage{bbm}
\usepackage{enumerate}
\usepackage{enumitem}
\usepackage{hyperref}

\newtheorem{thm}{Theorem}[section]

\newtheorem{prob}[thm]{Problem}
\newtheorem{prop}[thm]{Proposition}

\theoremstyle{definition}
\newtheorem{dfn}[thm]{Definition}

\renewcommand{\Re}{\mathrm{Re}}

\renewcommand{\le}{\leqslant}
\renewcommand{\ge}{\geqslant}

\newcommand{\Aes}{a_e}
\newcommand{\Aos}{a_o}
\newcommand{\Ae}{A_e}
\newcommand{\Ao}{A_o}
\newcommand{\cG}{\mathcal{G}}
\newcommand{\dG}{\vec{\mathcal{G}}}
\newcommand{\Z}{\mathbb{Z}}
\newcommand{\eps}{\varepsilon}
\newcommand{\Aut}{\mathrm{Aut}}
\newcommand{\1}{\mathbbm{1}}
\newcommand{\ST}{\mathrm{ST}}

\newcommand{\Stab}{\mathrm{Stab}}
\newcommand{\vspan}{\mathrm{span}}

\newcommand*{\Rom}[1]{{\uppercase\expandafter{\romannumeral#1\relax}}}
\newcommand{\tI}{\Rom{1}}
\newcommand{\tIp}{\tI$(p)$}
\newcommand{\tII}{\Rom{2}}
\newcommand{\tIII}{\Rom{3}}

\title{The number of additive triples in subsets of abelian groups}

\date{\today}

\author{Wojciech Samotij}
\address{School of Mathematical Sciences, Tel Aviv University, Tel Aviv 6997801, Israel}
\email{samotij@post.tau.ac.il}

\author{Benny Sudakov}
\address{Department of Mathematics, ETH, 8092 Z\"urich, Switzerland}
\email{benjamin.sudakov@math.ethz.ch}

\thanks{Research supported in part by: (WS) Institute for Mathematical Research (FIM), ETH Z\"urich and Israel Science Foundation grant 1147/14; (BS) SNSF grant 200021-149111.}

\begin{document}

\begin{abstract}
  A set of elements of a finite abelian group is called sum-free if it contains no Schur triple, i.e., no triple of elements $x,y,z$ with $x+y=z$. The study of how large the largest sum-free subset of a~given abelian group is had started more than thirty years before it was finally resolved by Green and Ruzsa a decade ago. We address the following more general question. Suppose that a set $A$ of elements of an abelian group $G$ has cardinality $a$. How many Schur triples must $A$ contain? Moreover, which sets of $a$ elements of $G$ have the smallest number of Schur triples? In this paper, we answer these questions for various groups $G$ and ranges of~$a$.
\end{abstract}

\maketitle

\section{Introduction}

\label{sec:introduction}

A typical problem in extremal combinatorics has the following form: What is the largest size of a structure which does not contain any forbidden configurations? Once this extremal value is known, it is very natural to ask how many forbidden configurations one is guaranteed to find in every structure of a certain size that is larger than the extremal value. There are many results of this kind. Most notably, there is a very large body of work on the problem of determining the smallest number of $k$-vertex cliques in a graph with $n$ vertices and $m$~edges, attributed to Erd{\H{o}}s and Rademacher; see~\cite{Er62, Er69, ErSi83, LoSi83, Ni11, Ra08, Re12}. In extremal set theory, there is an extension of the celebrated Sperner's theorem, where one asks for the minimum number of chains in a family of subsets of $\{1, \ldots, n\}$ with more than $\binom{n}{\lfloor n/2 \rfloor}$ members; see~\cite{DaGaSu13, DoGrKaSe13, ErKl74, Kl68}. Another example is a recent work in \cite{DaGaSu14}, motivated by the classical theorem of Erd\H{o}s, Ko, and Rado. It studies how many disjoint pairs must appear in a $k$-uniform set system of a certain size.

Analogous questions have been studied in the context of Ramsey theory. Once we know the maximum size of a~structure which does not contain some unavoidable pattern, we may ask how many such patterns are bound to appear in every structure whose size exceeds this maximum. For example, a well-known problem posed by Erd{\H{o}}s is to determine the minimum number of monochromatic $k$-vertex cliques in a $2$-colouring of the edges of~$K_n$; see, e.g.,~\cite{Co12, FrRo93, Th89}. This may be viewed as an extension of Ramsey's theorem. Another example is an extension of the famous theorem of Erd\H{o}s and Szekeres~\cite{ErSz35}, which states that any sequence of more than $k^2$ numbers contains a monotone (that is, monotonically increasing or monotonically decreasing) subsequence of length $k+1$. Here, one may ask what the  minimum number of monotone subsequences of length $k+1$ contained in a sequence of $n$ numbers is; see \cite{BaHuLiPiUdVo, My02, SaSu}.

In this paper, we consider a similar Erd\H{o}s--Rademacher-type generalisation of a classical problem in additive combinatorics. Recall that a \emph{Schur triple} in an abelian group $G$ is a triple of elements $x,y,z$ of $G$, not necessarily distinct, satisfying $x+y=z$. A set $A$ of elements of $G$ is called \emph{sum-free} if it contains no Schur triples. The study of sum-free sets in abelian groups goes back to the work of Erd\H{o}s~\cite{E65}. In 1965, he proved  that any set of $n$ non-zero integers contains a sum-free subset of size at least $n/3$ and asked whether the fraction $1/3$ could be improved. Despite significant interest in this problem, the matching upper bound was proved only recently by Eberhard, Green, and Manners~\cite{EbGrMa14}, who constructed a sequence of sets showing that Erd\H{o}s' result is asymptotically tight.  

A related question, which is also more than forty years old, is to determine how large the largest sum-free subset of a given finite abelian group is. The following two simple observations provide strong lower bounds for this quantity. First, note that by considering the `middle' interval of an appropriate length, one sees that the cyclic group $\Z_m$ contains a sum-free subset with $\lfloor\frac{m+1}{3}\rfloor$ elements. Second, if $G$ is an abelian group, $H$ is a subgroup of $G$, $\pi \colon G \to G / H$ is the canonical homomorphism, and $B$ is a sum-free subset of $G/H$, then the set $\pi^{-1}(B) \subseteq G$ is also sum-free. The appearance of the expression $\lfloor \frac{m+1}{3} \rfloor$ above explains why the following nomenclature is commonly used in this context.

\begin{dfn}
  Let $G$ be an abelian group of order $n$. We say that $G$ is of: (i) \emph{type~\tI} if $n$ has a prime factor $p$ satisfying $p \equiv 2 \pmod 3$; (ii) \emph{type~\tII} if $n$ has no such prime factor but $3$ divides $n$; (iii) \emph{type~\tIII} otherwise, i.e., if each prime factor $p$ of $n$ satisfies $p \equiv 1 \pmod 3$.
\end{dfn}

Using the above two observations, one can check that if $G$ is an abelian group with $n$ elements, then the largest sum-free set in $G$ has size at least
\begin{itemize}
\item
  $(\frac{1}{3}+\frac{1}{3p})n$ if $G$ is of type~\tI\ and $p$ is the smallest prime factor of $n$ with $p \equiv 2 \pmod 3$,
\item
  $\frac{n}{3}$ if $G$ is of type~\tII,
\item
  $(\frac{1}{3}-\frac{1}{3m})n$ if $G$ is of type~\tIII\ and $m$ is the largest order of an element in $G$.
\end{itemize}
It turns out that these simple lower bounds are actually tight, but the task of showing that this is indeed the case took more than thirty five years. This was first proved by Diananda and Yap~\cite{DiYa69} for groups of types \tI\ and \tII\ and in~\cite{RS, Y1, Y2} for some groups of type~\tIII. Only many years later, Green and Rusza~\cite{GrRu05} established it for all groups.

Motivated by these results on the size of the largest sum-free sets, we consider the following more general questions.

\begin{prob}
  \label{prob:main}
  Let $A$ be an $a$-element subset of a finite abelian group $G$. How many Schur triples must $A$ contain? Which sets of $a$ elements of $G$ have the minimum number of Schur triples?
\end{prob}

In this paper, we answer these questions for various groups $G$ and ranges of $a$. Some estimates for the number of Schur triples in large subsets of abelian groups appeared already in~\cite{GrRu05, LeLuSc01}, but to the best of our knowledge, we are the first to explicitly consider these questions and obtain exact results.

Given a subset $A$ of an abelian group, we shall denote by $\ST(A)$ the number of Schur triples contained in $A$. More precisely, we let
\[
\ST(A) = \left|\left\{ (x,y,z) \in A^3 \colon x+y=z \right\}\right|,
\]
so that if $x+y=z$ and $x \neq y$, then we consider  $(x,y,z)$ and $(y,x,z)$ as different triples.

Our first result concerns cyclic groups of prime order. In this case, we derive a complete answer to both parts of Problem~\ref{prob:main} from a classical result of Pollard~\cite{Po74} and its stability counterpart due to Nazarewicz, O'Brien, O'Neill, and Staples~\cite{NaOBrONeSt07}.

\begin{thm}
  \label{thm:Zp}
  Suppose that $p$ is an odd prime and order the elements of the $p$-element cyclic group $\Z_p$ as $x_1, \ldots, x_p$, where
  \[
  x_{2i} = \frac{p-1}{2} + i \qquad \text{and} \qquad x_{2i+1} = \frac{p-1}{2} - i.
  \]
  For every $A \subseteq \Z_p$ with $a$ elements,
  \begin{equation}
    \label{eq:STA-Zp}
    \ST(A) \ge \ST(\{x_1, \ldots, x_a\}) =
    \begin{cases}
      0, & \text{if $a \le \frac{p+1}{3}$,} \\
      \left\lfloor \frac{3a-p}{2} \right\rfloor \left\lceil \frac{3a-p}{2} \right\rceil, & \text{if $a > \frac{p+1}{3}$}.
    \end{cases}
  \end{equation}
  Moreover, if $\ST(\{x_1, \ldots, x_a\}) > 0$, then equality holds above only if $A = \varphi(\{x_1, \ldots, x_a\})$ for some $\varphi \in \Aut(\Z_p)$, that is, if $A = \xi \cdot \{x_1, \ldots, x_a\}$ for some nonzero $\xi \in \Z_p$.
\end{thm}

Our second result concerns groups of type~\tI. We shall say that a group $G$ is of \emph{type \tIp} if $p$ is the smallest prime factor of $|G|$ among those satisfying $p \equiv 2 \pmod 3$. Suppose that $G$ is of type \tI$(p)$. It was proved by Diananda and Yap~\cite{DiYa69} that the largest sum-free set in $G$ has $(\frac{1}{3} + \frac{1}{3p})|G|$ elements. We shall generalise this result by answering both questions in Problem~\ref{prob:main} under the assumption that $|A| \le (\frac{1}{3} + \frac{1+\delta}{3p})|G|$ for some absolute constant $\delta$.

\begin{thm}
  \label{thm:type-I}
  There exists a positive constant $\delta$ such that the following holds. Suppose that $p$ is a prime satisfying $p \equiv 2 \pmod 3$ and let $G$ be a group of type \tIp. If $0 \le t \le \delta|G|/p$, then for every $A \subseteq G$ with $(\frac{1}{3} + \frac{1}{3p})|G| + t$ elements,
  \begin{equation}
    \label{eq:type-I}
    \ST(A) \ge \frac{3t|G|}{p} + \1[p \neq 2] \cdot t^2.
  \end{equation}
\end{thm}

Our proof of Theorem~\ref{thm:type-I} will also yield the following characterisation of all sets achieving equality in~\eqref{eq:type-I}. Let $p$ and $G$ be as in the statement of the theorem and suppose that $p = 3k+2$. Let $\varphi \colon G \to \Z_p$ be an arbitrary surjective homomorphism, let $A_0 = \varphi^{-1}(\{k+1, \ldots, 2k+1\})$, and note that $A_0$ is a sum-free set with $(\frac{1}{3} + \frac{1}{3p})|G|$ elements. Given a $t$ with $0 \le t \le 2|G|/(7p)$, let $A_t'$ be an arbitrary sum-free subset of $\varphi^{-1}(\{k\})$ with $t$ elements\footnote{Such a set exists as if $k > 0$, then the set $\varphi^{-1}(\{k\})$ itself is sum-free and has $|G|/p$ elements; if $k = 0$, then $\varphi^{-1}(\{k\}) $ is a subgroup of $G$ with index $2$ and every nontrivial abelian group $H$ contains a sum-free set with at least $2|H|/7$ elements.} and let $A_t = A_0 \cup A_t'$. Then $\ST(A_t) = \frac{3t|G|}{p} + \1[p \neq 2] \cdot t^2$. Moreover, every set $A \subseteq G$ that achieves equality in~\eqref{eq:type-I} is of this form.

Our third result is a complete solution to Problem~\ref{prob:main} for the `hypercube', i.e., the group~$\Z_2^n$. Here, there is a very elegant way of describing a sequence of sets minimising the number of Schur triples among all subsets of $\Z_2^n$ of the same cardinality.

\begin{thm}
  \label{thm:Z_2n}
  Let $n$ be a positive integer. For each $a \in \{1, \ldots, 2^n-1\}$, let $A_a$ be the set of vectors in $\{0,1\}^n$ that are binary representations of the numbers $2^n-1, \ldots, 2^n-a$. Let $k$ be the unique integer satisfying $2^n - 2^k \le a < 2^n - 2^{k-1}$. Then for every $A \subseteq \Z_2^n$ with $a$ elements,
  \begin{equation}
    \label{eq:Z_2n}
    \ST(A) \ge \ST(A_a) = (3a+2^k-2^{n+1})(2^n-2^k).
  \end{equation}
\end{thm}

Our proof of Theorem~\ref{thm:Z_2n} will also yield the following characterisation of sets achieving equality in~\eqref{eq:Z_2n}. Let $a$, $k$, and $n$ be as in the statement of the theorem. For every $a$-element $A \subseteq \Z_2^n$ satisfying $\ST(A) = \ST(A_a)$, the following holds. There is a subgroup $K < \Z_2^n$ with $2^k$ elements such that $\Z_2^n \setminus K \subseteq A$ and $A \cap K$ is sum-free. One may check that each $A$ of this form satisfies $\ST(A) = \ST(A_a)$. As clearly each such $K$ is isomorphic to~$\Z_2^k$, one may say a little more about the structure of $A \cap K$ for certain ranges of $a$. In particular, it was proved in~\cite{ClDuRo90, ClPe92} that each sum-free subset of $\Z_2^k$ with more than $5 \cdot 2^{k-4}$ elements is contained in some maximum-size sum-free subset of $\Z_2^k$, i.e., the odd coset of some subgroup of index two. (Smaller sum-free sets of $\Z_2^k$ do not admit such an elegant structural description. For example, if $k \ge 4$, then the set $\{e_1, e_2, e_3, e_4, e_1+e_2+e_3+e_4\} + \vspan\{e_5, \ldots, e_k\}$, where $e_1, \ldots, e_k$ is a basis of $\Z_2^k$ as a vector space over $\Z_2$, is sum-free, has $5 \cdot 2^{k-4}$ elements, and is not contained in any maximum-size sum-free subset of $\Z_2^k$.)

Finally, we consider Problem~\ref{prob:main} for groups of type~\tII, i.e., groups whose order is divisible by three. As it turns out, here the answer is much less `uniform' among all groups in this class. To be more precise, given an abelian group $G$, let $f_G$ be the function defined by
\begin{equation}
  \label{eq:fG}
  f_G(a) = \min\{\ST(A) \colon A \subseteq G \text{ and } |A| = a\}
\end{equation}
and let $a_G$ be the largest cardinality of a sum-free set in $G$, i.e., $a_G = \max\{a \colon f_G(a) = 0\}$. On the one hand, if $G = \Z_3^n$, then $f_G$ `behaves' similarly as in the case when $G$ is of type \tIp\ for some fixed prime $p$, namely, $f_G(a+1) - f_G(a)$ is of order $|G|$ for all $a$ in an interval of length $\Omega(|G|/p)$ starting at $a_G$. On the other hand, if $G = \Z_3 \times \Z_p$, where $p$ is a prime with $p \equiv 1 \pmod 3$, then $f_G(a)$ is merely of order $(a-a_G)^2$ for all $a \ge a_G$. (As $G$ is of type~\tII, $a_G = p$.)

\begin{thm}
  \label{thm:Z_3n}
  There exists a positive constant $\delta$ such that the following holds. Let $n$ be a positive integer and suppose that $0 \le t \le \delta 3^{n-1}$. Then for every $A \subseteq \Z_3^n$ with $3^{n-1}+t$ elements,
  \begin{equation}
    \label{eq:Z_3n}
    \ST(A) \ge 3^{n-1}t + t^2.
  \end{equation}
  Moreover, \eqref{eq:Z_3n} holds with equality when $A$ is the union of $\{x \in \Z_3^n \colon x_1 = 1\}$ and an arbitrary $t$-element sum-free subset of $\{x \in Z_3^n \colon x_1 = 2\}$.
\end{thm}

\begin{prop}
  \label{prop:Z3Zp}
  Let $p$ be a prime. Then for every $a \in \{p+1, \ldots, 3p\}$, there exists an $a$-element set $A \subseteq \Z_3 \times \Z_p$ with
  \[
  \ST(A) \le 21(a-p)^2.
  \]
\end{prop}

Last but not least, the argument we use in our proof of Theorem~\ref{thm:type-I} can be adapted to show that every sufficiently large set $A$ of elements of an arbitrary finite abelian group that is nearly sum-free must necessarily contain a genuinely sum-free set $B$ such that $|A \setminus B|$ is very small.

\begin{prop}
  \label{prop:Green-Ruzsa-optimal}
  Suppose that $\eps > 0$ and let $G$ be a finite abelian group. If some $A \subseteq G$ has at least $(\frac{1}{3}+\eps)|G|$ elements and $\ST(A) \le \eps^2|G|^2/2$, then $A$ contains a sum-free set $B$ with $|A \setminus B| \le \eps|G|$.
\end{prop}

A slightly weaker version of this useful fact was first proved by Green and Ruzsa~\cite{GrRu05} (their statement has a stronger requirement on the number of Schur triples). Observe that our assumption on $\ST(A)$ in the above proposition is optimal up to an absolute multiplicative constant. Indeed, Proposition~\ref{prop:Z3Zp} implies that for every $\eps \in (0,\frac{2}{3})$, there are a group $G$ with no sum-free set larger than $|G|/3$ and a set $A$ with $\ST(A) \le 22 \eps^2 |G|^2$ and more than $(\frac{1}{3} + \eps)|G|$ elements (and hence no sum-free subset $B$ with $|A \setminus B| \le \eps |G|$).

\section{Cyclic groups of prime order}

\label{sec:Zp}

In this section, we prove Theorem~\ref{thm:Zp}. The first part of the theorem, a lower bound on the number of Schur triples in an arbitrary set of $a$ elements of $\Z_p$ is a fairly straightforward consequence of the following result of Pollard~\cite{Po74}, which generalises the well-known theorem of Cauchy~\cite{Ca13} and Davenport~\cite{Da35}.

\begin{thm}
  \label{thm:Pollard}
  Let $p$ be a prime and let $A, B \subseteq \Z_p$ . For an integer $r$, denote by $N_r$ the number of elements of $\Z_p$ which are expressible in at least $r$ ways as $x+y$ with $x \in A$ and $y \in B$. Then for every $r$ with $1 \le r \le \min\{|A|,|B|\}$,
  \begin{equation}
    \label{eq:Pollard}
    N_1 + \ldots + N_r \ge r \cdot \min\{p, |A|+|B|-r\}.
  \end{equation}
\end{thm}

The second part of Theorem~\ref{thm:Zp} will be derived from the following stability counterpart of Pollard's result due to Nazarewicz, O'Brien, O'Neill, and Staples~\cite{NaOBrONeSt07}.

\begin{thm}
  \label{thm:Pollard-stability}
  Let $p$ be a prime, let $A, B \subseteq \Z_p$, and let $r$ be an integer with $1 \le r \le \min\{|A|,|B|\}$. Then equality holds in~\eqref{eq:Pollard} of Theorem~\ref{thm:Pollard} if and only if at least one of the following conditions holds:
  \begin{enumerate}
    \renewcommand{\theenumi}{\textit{(\roman{enumi})}}
  \item
    \label{item:Pollard-stability-i}
    $\min\{|A|, |B|\} = r$,
  \item
    \label{item:Pollard-stability-ii}
    $|A| + |B| \ge p + r$,
  \item
    \label{item:Pollard-stability-iii}
    $|A| = |B| = r + 1$ and $B = x - A$ for some $x \in \Z_p$, or
  \item
    \label{item:Pollard-stability-iv}
    $A$ and $B$ are arithmetic progressions with the same common difference.
  \end{enumerate}
\end{thm}

\begin{proof}[{Proof of Theorem~\textup{\ref{thm:Zp}}}.]
  Fix an arbitrary set $A$ of $a$ elements of $\Z_p$. Given an $r \ge 1$, let $S_r$ denote the set of all elements in $\Z_p$ that are expressible in at least $r$~ways as $x+y$ with $x, y \in A$. Recall that $N_r=|S_r|$. Moreover, let $N_r' = |S_r \cap A|$. Clearly, $N_r' \ge N_r + a - p$ and hence by Theorem~\ref{thm:Pollard}, for any $R \ge 0$,
  \begin{equation}
    \label{eq:STA}
    \ST(A) = \sum_{r \ge 1} N_r' \ge \sum_{r = 1}^R N_r' \ge \sum_{r=1}^R N_r + R(a-p) \ge R \cdot \big( \min\{p, 2a-R\} + a - p \big).
  \end{equation}
  Let $R = \max\big\{0, \lceil \frac{3a-p}{2} \rceil\big\}$. Note that $a \le p$ implies that $2a \le p + \frac{3a-p}{2}$ and consequently
  \[
  2a \le p + \left\lceil \frac{3a-p}{2} \right\rceil \le p + R.
  \]
  In particular, the minimum in the right-hand side of~\eqref{eq:STA} is equal to $2a-R$ and therefore,
  \begin{equation}
    \label{eq:STA-final}
      \ST(A) \ge R(3a-R-p) = \max\left\{0, \left\lceil \frac{3a-p}{2} \right\rceil\right\} \cdot \min\left\{ 3a-p, \left\lfloor \frac{3a-p}{2} \right\rfloor \right\}.
  \end{equation}
  It is straightforward to check that the right-hand side of~\eqref{eq:STA-final} is equal to the right-hand side of~\eqref{eq:STA-Zp}. In order to complete the proof of~\eqref{eq:STA-Zp}, we still need to verify the equality there. Assume first that $a$ is even. Then $\{x_1, \ldots, x_a\} = \big\{ \frac{p+1-a}{2}, \ldots, \frac{p+a-1}{2}\big\}$. A Schur triple $(x,y,z)$ in $\{x_1, \ldots, x_a\}$ may be of one of the following two types: either $x+y=z$ or $x+y = z+p$, where the equalities hold in $\Z$. Let $R' = \max\big\{0, \frac{3a-p-1}{2}\big\} = \max\big\{0, \lfloor\frac{3a-p}{2}\rfloor\big\}$. It is easy to check that $\frac{p+1-a}{2}$ plays the role of $x$ in exactly $R'$ triples of the first type, as $y$ ranges over the $R'$ smallest elements of the interval $\big\{ \frac{p+1-a}{2}, \ldots, \frac{p+a-1}{2}\big\}$. More generally, the element $\frac{p+1-a}{2} + i$ plays the role of $x$ in exactly $R'-i$ triples of the first type. By symmetry, $\frac{p+a-1}{2}$ plays the role of $x$ in exactly $R'$ triples of the second type, as $y$ ranges over the $R'$ largest elements of the interval $\big\{ \frac{p+1-a}{2}, \ldots, \frac{p+a-1}{2}\big\}$, and more generally, $\frac{p+a-1}{2} - i$ plays the role of $x$ in exactly $R'-i$ triples of the second type. It follows that
  \begin{equation}
    \label{eq:STxi-a-even}
    \begin{split}
      \ST(\{x_1, \ldots, x_a\}) & = 2 \cdot \sum_{r = 1}^{R'} r = 2 \cdot \binom{R'+1}{2} = R'(R'+1) \\
      & = \max\left\{0, \left\lfloor \frac{3a-p}{2} \right\rfloor\right\} \cdot \max\left\{1, \left\lceil\frac{3a-p}{2}\right\rceil\right\}.
    \end{split}
  \end{equation}
  It is straightforward to verify that the right-hand side of~\eqref{eq:STxi-a-even} is equal to the right-hand side of~\eqref{eq:STA-Zp}. When $a$ is odd, then $\{x_1, \ldots, x_a\} = \big\{ \frac{p-a}{2}, \ldots, \frac{p+a-2}{2}\big\}$ and, letting $R'' = \max\big\{0, \frac{3a-p}{2}\big\} = \max\big\{0, \lfloor \frac{3a-p}{2} \rfloor\big\} = \max\big\{0, \lceil\frac{3a-p}{2}\rceil\big\}$, analogous considerations yield
  \begin{equation}
    \label{eq:STxi-a-odd}
    \begin{split}
      \ST(\{x_1, \ldots, x_a\}) & = \binom{R''+1}{2} + \binom{R''}{2} = (R'')^2 \\
      & = \max\left\{0, \left\lceil \frac{3a-p}{2} \right\rceil\right\} \cdot \max\left\{0, \left\lfloor \frac{3a-p}{2} \right\rfloor \right\}.      
    \end{split}
  \end{equation}
  It is easy to check that the right-hand side of~\eqref{eq:STxi-a-odd} is equal to the right-hand side of~\eqref{eq:STA-Zp}.

  We now characterise all sets $A$ with $a$ elements that achieve the lower bound in~\eqref{eq:STA-Zp} whenever the right-hand side of~\eqref{eq:STA-Zp} is nonzero, that is, when $a > \lfloor \frac{p+1}{3} \rfloor$. To this end, let us analyse when all inequalities in~\eqref{eq:STA} hold with equality. In particular, equality must hold in~\eqref{eq:Pollard} of Theorem~\ref{thm:Pollard} invoked with $A \leftarrow A$, $B \leftarrow A$, and $r \leftarrow R$, where $R = \max\big\{0, \lceil \frac{3a-p}{2} \rceil\big\} \ge 1$ and the inequality follows from our assumption that $a > \lfloor \frac{p+1}{3} \rfloor$. Theorem~\ref{thm:Pollard-stability} tells us that~\eqref{eq:STA} can hold with equality only if one of the following conditions is satisfied:
  \begin{enumerate}
    \renewcommand{\theenumi}{(\roman{enumi})}
  \item
    \label{item:Zp-stab-1}
    $a = R$,
  \item
    $2a \ge p + R$,
  \item
    $a = R + 1$ and $A = x - A$ for some $x \in \Z_p$,
  \item
    \label{item:Zp-stab-4}
    $A$ is an arithmetic progression.
  \end{enumerate}
  Now observe that
  \begin{align*}
    a = R=\lceil a-\frac{p-a}{2} \rceil & \Longleftrightarrow \left\lfloor \frac{p-a}{2} \right\rfloor = 0 \Longleftrightarrow a \ge p-1, \\
    2a \ge p+R & \Longleftrightarrow a \ge \left\lceil \frac{a+p}{2} \right\rceil \Longleftrightarrow a \ge p, \\
    a = R+1 & \Longleftrightarrow \left\lfloor \frac{p-a}{2} \right\rfloor = 1 \Longleftrightarrow a \in \{p-3, p-2\},
  \end{align*}
  and every set $A \subseteq \Z_p$ that satisfies either $|A| \ge p-2$ or $|A| = p-3$ and $A = x - A$ for some $x \in \Z_p$ is an arithmetic progression (to see this, note that $A^c$ is an arithmetic progression), we deduce that each of \ref{item:Zp-stab-1}--\ref{item:Zp-stab-4} implies that $A$ must be an arithmetic progression with common difference $d$. We may assume that $d = 1$ as otherwise we may replace $A$ by its automorphic image $d^{-1} \cdot A$. Hence, $A = \{x, \ldots, x+a-1\}$ for some $x \in \Z_p$.

  In order for~\eqref{eq:STA} to hold with equality, it must also be that
  \[
  |A \cap S_R| = N_R' = N_R + a - p = |S_R| - |A^c|,
  \]
  that is, $A^c \subseteq S_R$. One easily checks that $S_R = \{2x+R-1, \ldots, 2x +2a-R-1\}$ and hence $|S_R| = 2(a-R)+1 = 2\lfloor \frac{p-a}{2} \rfloor + 1$. If $a$ is even, then $|A^c| = p-a = |S_R|$ and hence $x + a = 2x+R-1$, which yields $x = \frac{p-a+1}{2}$, that is,
\[
A = \left\{ \frac{p+1}{2} - \frac{a}{2}, \ldots, \frac{p-1}{2} + \frac{a}{2} \right\} = \{x_1, \ldots, x_a\}.
\]
If $a$ is odd, then $|A^c| = |S_R| - 1$ and hence either $x+a = 2x+R-1$ or $x+a = 2x+R$, yielding $x = \frac{p-a}{2}$ or $x = \frac{p-a}{2} + 1$, that is,
\[
A = \pm\left\{ \frac{p-a}{2} + 1, \ldots, \frac{p+a}{2} \right\} = \pm\{x_1, \ldots, x_a\}.\qedhere
\]
\end{proof}

\section{Groups of type~\tI}

\label{sec:type-I}

In this section, we prove Theorem~\ref{thm:type-I}. Our argument uses some ideas from~\cite{GrRu05,LeLuSc01}. Actually, one may adapt the arguments of these two papers to establish Theorem~\ref{thm:type-I} under the stronger assumption that $t \le \delta n/p^4$ for some positive constant $\delta$. Our main tool will be the following classical result of Kneser~\cite{Kn53,Kn55}. The version stated below is~\cite[Theorem~3.1]{Ke60}. Recall that the \emph{stabiliser} of a set $A$ of elements of an abelian group $G$, denoted by $\Stab(A)$, is defined by
\[
\Stab(A) = \{x \in G \colon A + x = A\}.
\]

\begin{thm}
  \label{thm:Kneser}
  Let $A$ and $B$ be finite non-empty subsets of an Abelian group $G$ satisfying $|A+B| \le |A| + |B| - 1$. Then $H = \Stab(A+B)$ satisfies
  \[
  |A+B| = |A+H| + |B+H| - |H|.
  \]
\end{thm}

\begin{proof}[{Proof of Theorem~\textup{\ref{thm:type-I}}}.]
  Let $\delta = 1/82$ and let $p$, $G$, and $t$ be as in the statement of the theorem. Denote the order of $G$ by $n$ and let $A$ be an arbitrary set of $(\frac{1}{3} + \frac{1}{3p})n + t$ elements of $G$. Let $n' = n/(3p)$ and define
  \[
    C_- = \big\{x \in A \colon |(x-A) \cap A| \ge n'\big\} \quad \text{and} \quad
    C_+ = \big\{x \in A \colon |(x+A) \cap A| \ge n'\big\}.
  \]
  Using the inclusion-exclusion principle (Bonferroni's inequality) to count Schur triples $(x,y,z) \in A^3$ such that $x \in C_+$, $y \in C_+$, or $z \in C_-$ yields
  \[
  \ST(A) \ge (2|C_+| + |C_-|) \cdot n' - |C_+|^2 - 2|C_-||C_+|.
  \]
  Therefore, if $|C_-| \ge 4t$ and $|C_+| \ge 4t$, then, passing to subsets of $C_-$ and $C_+$ of size exactly $4t$, we have
  \[
  ST(A) \ge 12tn' - 48t^2 = \frac{4tn}{p} - 48t^2 > \frac{3tn}{p} + t^2,
  \]
  where the last inequality follows as $49t^2 \le 49\delta tn/p < tn/p$. Hence, we may assume that either $|C_-| < 4t$ or $|C_+| < 4t$.

  Given a $* \in \{-, +\}$, let $B = A \setminus C_*$ and observe that $|(x*A) \cap A^c| \ge |A| - n'$ for each $x \in B$ and hence for every $x, y \in B$,
  \begin{multline*}
    |(x*A) \cap (y * A)| \ge |(x*A) \cap (y*A) \cap A^c| \ge 2(|A|-n') - (n-|A|) \\
    = 3|A| - n - 2n' = n/p + 3t - 2n' = n/(3p) + 3t \ge n/(3p).
  \end{multline*}
  In other words, for every pair $x,y \in B$, there are at least $n/(3p)$ pairs $a,b \in A$ such that $x-y = a - b$.

  Fix a $* \in \{-,+\}$ such that $|C_*| < 4t$ and let $B = A \setminus C_*$. Using our observation above to count Schur triples $(x,y,z) \in A^3$ such that $x \in B - B$ yields
  \begin{equation}
    \label{eq:STA-B-B-lower}
    ST(A) \ge |(B - B) \cap A| \cdot n/(3p).
  \end{equation}
  Hence, we may assume that $|(B-B) \cap A| \le 9t + 3t^2p/n \le 10t$, where the last inequality holds as $t \le \delta n / p$. In particular,
  \begin{equation}
    \label{eq:B-B-upper}
    |B-B| \le n - |A| + 10t = 2|A|-n/p+7t \le 2|B| - n/p + 15t,
  \end{equation}
  where the last inequality follows as $|B| = |A| - |C_*| \ge |A| - 4t$. Letting $H = \Stab(B-B)$, Kneser's theorem (Theorem~\ref{thm:Kneser}) implies that
  \begin{equation}
    \label{eq:B-B-lower}
    |B-B| = 2|B+H|-|H| \ge 2|B| - |H|.
  \end{equation}
  Putting~\eqref{eq:B-B-upper} and~\eqref{eq:B-B-lower} together yields
  \begin{equation}
    \label{eq:H-properties}
    |H| \ge n/p -15t \quad \text{and} \quad |B+H| - |B| \le \frac{|H|-n/p+15t}{2}.
  \end{equation}
  Let $m = |G/H|$ and observe that~\eqref{eq:H-properties} yields
  \begin{equation}
    \label{eq:m-lower}
    m \le \frac{n}{n/p - 15t} \le \frac{p}{1-15\delta}.    
  \end{equation}
  Note also that by our assumption that $|C_*| < 4t$, we have
  \begin{equation}
    \label{eq:B-lower}
    |B| > |A|-4t = n/3+n/(3p)-3t \ge n/3 + (1/3 - 3\delta) n/p > n/3.
  \end{equation}
  Now, let $B_H = (B+H) / H$. As $|B_H| > m/3$ by~\eqref{eq:B-lower}, we must have $|B_H| \ge \lceil \frac{m+1}{3} \rceil$. We shall now consider two cases.

  \medskip
  \textsc{Case 1.} $m \not\equiv 2 \pmod 3$.
  \smallskip

  In this case, $|B_H| \ge \frac{m+2}{3}$ and it follows from~\eqref{eq:B-B-lower} that $|B-B|/|H| = 2|B_H| - 1 \ge \frac{2m+1}{3}$. This implies that
  \[
  |(B-B) \cap A| \ge \left(\frac{2}{3}+\frac{1}{3m}\right)n + \left(\frac{1}{3}+\frac{1}{3p}\right)n - n > \frac{n}{3p} \ge 10t,
  \]
  contradicting our assumption, cf.~\eqref{eq:STA-B-B-lower}.

  \medskip
  \textsc{Case 2.} $m \equiv 2 \pmod 3$.
  \smallskip
  
  Let $q$ be the smallest prime factor of $m$ satisfying $q \equiv 2 \pmod 3$; $m$ has such a prime factor as otherwise $m \not\equiv 2 \pmod 3$. Since $m$ divides $n$ and $p$ is the smallest prime factor of $n$ satisfying $p \equiv 2 \pmod 3$, we must have $q \ge p$. But $m < 2p$ by~\eqref{eq:m-lower}, so necessarily $m = q$, that is, $m$ is a prime satisfying $m \equiv 2 \pmod 3$ and $m \ge p$. This means that $G/H$ is the cyclic group $\Z_m$.

  We now claim that $B_H \cap (B_H - B_H) = \emptyset$. Indeed, otherwise we would have $|(B+H) \cap (B-B)| \ge |H|$ (since $B-B$ is a union of cosets of $H$) and, since $B \subseteq A$, by~\eqref{eq:H-properties},
  \begin{equation}
    \label{eq:A-structure}
    |(B+H) \setminus A| \le |B+H|-|B| \le \frac{n/m-n/p+15t}{2} \le 7.5t,
  \end{equation}
  which in turn would yield
  \[
  |A \cap (B-B)| \ge |H| - 7.5t = n/m - 7.5t \ge (1-15\delta)n/p - 7.5t > 10t,
  \]
  contradicting our assumption, cf.~\eqref{eq:STA-B-B-lower}.
  
  Therefore, it must be that $B_H \cap (B_H - B_H) = \emptyset$, that is, $B_H \subseteq \Z_m$ is a sum-free set. But $|B_H| > m/3$, which means that $|B_H| = \frac{m+1}{3}$ and it follows from the results of Diananda and Yap~\cite{DiYa69} that, up to isomorphism, $B_H = \{\ell+1, \ldots, 2\ell+1\}$, where $m = 3\ell+2$.

  \medskip

  Let $A_0 = B+H$ and let $\varphi \colon G \to \Z_m$ be a homomorphism that maps $A_0$ to $\{\ell+1, \ldots, 2\ell+1\}$; in particular $H = \varphi^{-1}(0)$ and $B_H = \{\ell+1, \ldots, 2\ell+1\}$. As $-\varphi$ is also such a homomorphism and $-\ell = 2\ell+2$ in $\Z_m$, we may assume that $|A \cap \varphi^{-1}(\ell)| \ge |A \cap \varphi^{-1}(2\ell+2)|$. It follows from~\eqref{eq:A-structure} that $|A_0 \setminus A| \le 7.5t$. We shall now perform a stability analysis of $A$ and prove that~\eqref{eq:type-I} holds and the inequality there is strict unless $A_0 \subseteq A$, $m = p$, and $A \setminus A_0$ is a sum-free subset of $\varphi^{-1}(\{\ell\})$.

  We first claim that replacing an element from $A \setminus A_0$ with an element of $A_0 \setminus A$ only decreases the number of Schur triples in $A$. Indeed, as $A_0$ is sum-free, a given element of $A_0$ participates only in Schur triples that contain an element of $A \setminus A_0$; clearly, the number of such triples is at most $6|A \setminus A_0|$, which is at most $51t$ as
  \[
  |A \setminus A_0| = |A| - |A \cap A_0| = (|A_0| + t) - (|A_0| - |A_0 \setminus A|) \le 8.5t.
  \]
  On the other hand, for every $j \not\in \{\ell+1, \ldots, 2\ell+1\}$, there are $h, i \in \{\ell+1, \ldots, 2\ell+1\}$ such that $h + i = j$. Therefore if $x \in A \setminus A_0$, then, letting $j = \varphi(x) \not\in \{\ell+1, \ldots, 2\ell+1\}$, the number $\ST_x(A)$ of Schur triples in $A$ that contain $x$ satisfies
  \begin{multline*}
    \ST_x(A) \ge \left|\big(x - (\varphi^{-1}(i) \cap A)\big) \cap (\varphi^{-1}(h) \cap A)\right| \ge |\varphi^{-1}(i) \cap A| + |\varphi^{-1}(h) \cap A| - |H| \\
    \ge 2(|H| - 7.5t) - |H| \ge n/m - 15t > (1-15\delta)n/p - 15t > 51t.
  \end{multline*}
  Therefore, it suffices to prove~\eqref{eq:type-I}, and characterise all cases of equality there, under the assumption that $A_0 \subseteq A$.

  Now, note that each of $\ell$ and $2\ell+2$ participates in exactly three Schur triples with two elements of $\{\ell+1, \ldots, 2\ell+1\}$, namely: $(\ell, \ell+1, 2\ell+1)$, $(\ell+1, \ell, 2\ell+1)$, and $(2\ell+1,2\ell+1,\ell)$ and $(\ell+1,\ell+1,2\ell+2)$, $(2\ell+1,2\ell+2,\ell+1)$, and $(2\ell+2,2\ell+1,\ell+1)$. On the other hand, every element of $\Z_m \setminus \{\ell, \ldots, 2\ell+2\}$ participates in at least four such triples. It follows that:
  \begin{enumerate}
    \renewcommand{\theenumi}{(\roman{enumi})}
  \item
    \label{item:stability-1}
    Every element of $\varphi^{-1}(\{\ell,2\ell+2\})$ forms $3n/m$ Schur triples with two elements of $A_0$ and at most $6|A \setminus A_0|$ additional Schur triples with two elements of $A$ (one of which is not in $A_0$).
  \item
    \label{item:stability-2}
    Every element of $\varphi^{-1}(\Z_m \setminus \{\ell, \ldots, 2\ell+2\})$ forms at least $4n/m$ Schur triples with two elements of $A_0$.
  \end{enumerate}
  Since $|A_0| = (\frac{1}{3} + \frac{1}{3m})n$, our assumption that $A_0 \subseteq A$ and~\eqref{eq:m-lower} imply that
  \[
  |A \setminus A_0| = \frac{n}{3p} + t - \frac{n}{3m} \le \left(\frac{1}{3} +\delta\right)\frac{n}{p} - \frac{n}{3m} \le \left[\frac{1}{1-15\delta}\left(\frac{1}{3} + \delta\right) - \frac{1}{3}\right] \frac{n}{m} < \frac{n}{6m}.
  \]
  It therefore follows from~\ref{item:stability-1} and~\ref{item:stability-2} that moving elements from $A \cap \varphi^{-1}(\Z_m \setminus \{\ell, \ldots, 2\ell+2\})$ to $\varphi^{-1}(\{\ell,2\ell+2\})$ only decreases $\ST(A)$. Therefore, we may restrict our attention to sets $A$ satisfying
  \[
  \varphi^{-1}(\{\ell+1, \ldots, 2\ell+1\}) = A_0 \subseteq A \subseteq A_0 \cup \varphi^{-1}(\{\ell, 2\ell+2\}) = \varphi^{-1}(\{\ell, \ldots, 2\ell+2\}).
  \]

  Observe that if $\ell > 0$, then every ordered pair of elements $(x,y) \in \varphi^{-1}(\ell)^2 \cup \varphi^{-1}(2\ell+2)^2$ satisfying $(x,y) \in A^2$ participates in a unique Schur triple (in $A$) in which $x$ precedes $y$, the triple $(x,y,x+y)$. On the other hand, if $\ell > 0$, then every pair $(x,y) \in \varphi^{-1}(\ell) \times \varphi^{-1}(2\ell+2)$ satisfying $(x,y) \in A^2$ participates in four Schur triples in $A$: the triples $(x,y-x,y)$, $(y-x,x,y)$, $(y,x-y,x)$, and $(x-y,y,x)$. Counting separately Schur triples in $A$ that contain two (using~\ref{item:stability-1}), one (using the above observation), and no elements of $A_0$ yields
  \[
  \begin{split}
    \ST(A) & \ge |A \setminus A_0| \cdot \frac{3n}{m} + \1[m > 2] \cdot |A \setminus A_0|^2 + \ST(A \setminus A_0) \\
    & = \left( \frac{n}{3p} + t - \frac{n}{3m} \right) \cdot \frac{3n}{m} + \1[m > 2] \cdot \left( \frac{n}{3p} + t - \frac{n}{3m} \right)^2 + \ST(A \setminus A_0) \\
    & \ge \frac{3nt}{p} + \1[p > 2] \cdot t^2 + \ST(A \setminus A_0),
  \end{split}
  \]
  where the first inequality is strict unless $\ell = 0$, $A \cap \varphi^{-1}(\ell) = \emptyset$, or $A \cap \varphi^{-1}(2\ell+2) = \emptyset$ and the last inequality is strict unless $m = p$ (recall that $p \le m \le 2p$). This completes the proof.
\end{proof}

\section{The hypercube $\Z_2^n$}

\label{sec:Z_2n}

In this section, we prove Theorem~\ref{thm:Z_2n}. One way of obtaining lower bounds on $\ST(A)$ in our proof will be using eigenvalue analysis of the \emph{Cayley graph} of $G$ generated by $A$. Recall that given an abelian group $G$ and an $A \subseteq G \setminus \{0\}$ satisfying $A = -A$, we define $\cG_A$ to be the graph with vertex set $G$ whose edges are all pairs $\{x,y\}$ such that $x-y \in A$. It follows from this definition that for every $A \subseteq G \setminus \{0\}$ with $A = -A$,
\begin{equation}
  \label{eq:STA-ecGA}
  \ST(A) = 2e(\cG_A[A]),
\end{equation}
where $\cG_A[A]$ denotes the subgraph of $\cG_A$ induced by $A$. We shall derive lower bounds on $e(\cG_A[A])$ using the following well-known result of Alon and Chung~\cite{AlCh88}.

\begin{thm}
  \label{thm:AlCh}
  Let $\cG$ be an $N$-vertex $D$-regular graph and let $\lambda$ be the smallest eigenvalue of its adjacency matrix. Then for every $U \subseteq V(\cG)$,
  \[
  2e(\cG[U]) \ge \frac{D}{N} |U|^2 + \frac{\lambda}{N} |U| (N - |U|).
  \]
\end{thm}

Precise eigenvalue analysis of $\cG_A$ will be possible in our setting due to the fact that the \emph{characters} of any abelian group $G$ form a basis of  eigenvectors of $\cG_A$ for every $A$. Moreover, in the case $G = \Z_2^n$, there is a one-to-one correspondence between nontrivial characters of $G$ and subgroups of $G$ with index two. More precisely, for each nontrivial character $\chi \in \hat{G}$ there is a subgroup $H < G$ of index $2$ such that
\[
\chi(x) =
\begin{cases}
  1, & \text{if $x \in H$}, \\
  -1, & \text{if $x \not\in H$}.
\end{cases}
\]
In particular, the smallest eigenvalue $\lambda$ of $\cG_A$ satisfies
\[
\lambda = \min\big\{|A \cap H| - |A \cap H^c| \colon \text{$H < G$ with $[G:H]=2$} \big\}.
\]

\begin{proof}[{Proof of Theorem~\textup{\ref{thm:Z_2n}}}.]
  Let $n$ be a positive integer, let $G = \Z_2^n$, and let $a$ and $k$ be as in the statement of the theorem.  We may assume that $k \le n-1$ as otherwise both~\eqref{eq:Z_2n} and the characterisation of sets achieving equality are vacuous. We shall first count Schur triples in the set $A_a$ and establish the equality in~\eqref{eq:Z_2n}. Given an $x \in G$, let us denote by $\overline{x}$ the integer with binary representation $x$ (viewed as a $\{0,1\}$-vector), so that
  \[
  \{\overline{x} \colon x \in A_a \} = \{2^n-1, \ldots, 2^n-a\}.
  \]
  Fix some $x \in A_a$ and let $j$ be the unique integer such that $2^j \le \overline{x} < 2^{j+1}$. We claim that the number $\ST_x(A_a)$ of (ordered) Schur triples containing $x$ such that $\overline{x}$ is the smallest element is $3(2^n-2^{j+1})$. (As $0 \not\in A_a$, each Schur triple in $A_a$ contains three distinct nonzero elements.) It is enough to show that the number of pairs $\{y,z\} \subseteq A_a$ satisfying $x + y + z = 0$ and $\overline{x} < \overline{y} < \overline{z}$ is $2^{n-1} - 2^j$. To this end, note first that for every such pair, $\overline{z} \ge 2^{j+1}$ since otherwise $2^j \le \overline{x}, \overline{y}, \overline{z} < 2^{j+1}$ and then $\overline{x+y+z} \ge 2^j$; consequently, also $\overline{y} \ge 2^{j+1}$ as otherwise $\overline{x+y+z} \ge 2^{j+1}$. Conversely, given an arbitrary element $z$ such that $\overline{z} \ge 2^{j+1}$, we have $\overline{z+x} \ge 2^{j+1} > \overline{x}$ and hence $\{z, z+x\} \subseteq A_a$ is such a pair. Thus the number of these pairs is $\frac{1}{2}(2^n-2^{j+1})$, as claimed. 

  Now, recall that $k$ satisfies $2^n - 2^k \le a < 2^n - 2^{k-1}$. We may assume that $k \le n-1$ as otherwise $A_a \subseteq A_{2^{n-1}}$ is sum-free.  Let $t = a - 2^n + 2^k$. Then $0 \le t < 2^{k-1}$ and
  \[
  \begin{split}
    \ST(A_a) & = \sum_{x \in A_a} \ST_x(A_a) = 3 \cdot \left[\sum_{j = k}^{n-1} 2^j (2^n-2^{j+1}) + t (2^n - 2^k)\right] \\
    & = 3(2^n-2^k)2^n - 2(4^n-4^k) + 3(a-2^n+2^k)(2^n-2^k) \\
    & = (3a - 2^{n+1} + 2^k)(2^n-2^k).
  \end{split}
  \]

  Let us now fix an arbitrary $a$-element set $A \subseteq G$. We shall prove that $\ST(A) \ge \ST(A_a)$ by induction on $n$. We may assume that $0 \not\in A$ as one may easily check that replacing $0$ with an arbitrary element of $G \setminus A$ decreases the number of Schur triples by at least one (as $x + x \neq x$ unless $x = 0$). The case $n=1$ is trivial, including the characterisation of sets achieving equality in~\eqref{eq:Z_2n}, so let us assume that $n \ge 2$. Let $H < G$ be the subgroup of index $2$ that minimises $|A \cap H|$, let $\Ae = A \cap H$ (the set of `even' elements of $A$), and let $\Ao = A \cap H^c$ (the set of `odd' elements of $A$). Moreover, let $\Aes = |\Ae|$ and $\Aos = |\Ao|$.

  Observe that each Schur triple in $A$ contains an even number of `odd' elements (elements of $\Ao$). As $H \cong \Z_2^{n-1}$, the number of Schur triples that contain only elements of $\Ae$ satisfies
  \begin{equation}
    \label{eq:ST-Ae}
    \ST(\Ae) \ge \ST^{(n-1)}\left(A_{\Aes}^{(n-1)}\right),
  \end{equation}
  where the superscript `$(n-1)$' signifies the fact that we are referring to elements and Schur triples in $\Z_2^{n-1}$. The number of Schur triples that contain one element of $\Ae$ and two elements of $\Ao$ may be estimated as follows. Fix some $x \in \Ae$. The number of Schur triples in $A$ that contain $x$ and two elements of $\Ao$ is precisely $3|(x+\Ao) \cap \Ao|$, as the elements of $(x+\Ao) \cap \Ao$ are in one-to-one correspondence with ordered pairs $(y,z) \in \Ao^2$ such that $x+y+z=0$. Since $x+\Ao, \Ao \subseteq H^c$, then
  \begin{equation}
    \label{eq:AoxAo}
    |(x+\Ao) \cap \Ao| \ge |x+\Ao| + |\Ao| - |H^c| = 2|\Ao| - |H| = 2\Aos - 2^{n-1}
  \end{equation}
  and consequently, recalling~\eqref{eq:ST-Ae},
  \begin{equation}
    \label{eq:STA-lower}
    \ST(A) \ge \ST(\Ae) + 3\Aes(2\Aos - 2^{n-1}) \ge \ST^{(n-1)}\left(A_{\Aes}^{(n-1)}\right) + 3\Aes(2(a-\Aes) - 2^{n-1}).
  \end{equation}

  We claim that the right-hand side of~\eqref{eq:STA-lower} is greater than or equal to $\ST(A_a)$ as long as
  \begin{equation}
    \label{eq:Aes-case1}
    \Aes < \max\left\{ 2^{n-1}-2^{k-1} , a - 2^{n-1} + 2^{k-2} \right\},
  \end{equation}
  and equality holds only when $\Aes = a - 2^{n-1}$ or when $\Aes = 2^{n-1} - 2^{k-1}$ and $a > 2^n - 2^{k-1} - 2^{k-2}$.

  To see this, note first that $\Aes \ge a - |H^c| = a - 2^{n-1} \ge 2^{n-1} - 2^k$. If moreover $\Aes < 2^{n-1} - 2^{k-1}$, then by our inductive assumption,
  \begin{equation}
    \label{eq:STAe-lower}
    \ST^{(n-1)}\left(A_{\Aes}^{(n-1)}\right) = (3\Aes + 2^k - 2^n)(2^{n-1}-2^k).
  \end{equation}
  Substituting the right-hand side of~\eqref{eq:STAe-lower} into~\eqref{eq:STA-lower}, we verify that if $a - 2^{n-1} \le \Aes < 2^{n-1} - 2^{k-1}$, then the right-hand side of~\eqref{eq:STA-lower} is at least as large as $\ST(A_a)$, with equality holding only if $\Aes = a - 2^{n-1}$. To see this, observe that the difference of these functions is
  \[
  6\left(\Aes-2^{n-1}+2^{k-1}\right) \left(a-\Aes-2^{n-1}\right),
  \]
  which is quadratic in~$\Aes$, with the coefficient of $\Aes^2$ negative, and equal to zero if $\Aes = a - 2^{n-1}$ or $\Aes = 2^{n-1} - 2^{k-1}$.

  Assume now that $2^{n-1} - 2^{k-1} \le \Aes < a - 2^{n-1} + 2^{k-2}$. In particular, $\Aes < 2^{n-1} - 2^{k-2}$ and hence by the inductive assumption,
  \begin{equation}
    \label{eq:STAe-lower-2}
    \ST^{(n-1)}\left(A_{\Aes}^{(n-1)}\right) = (3\Aes + 2^{k-1} - 2^n)(2^{n-1}-2^{k-1}).
  \end{equation}
  Substituting the right-hand side of~\eqref{eq:STAe-lower-2} into \eqref{eq:STA-lower}, we verify that if $2^{n-1} - 2^{k-1} \le \Aes < a - 2^{n-1} + 2^{k-2}$, then the right hand side of~\eqref{eq:STA-lower} is at least as large as $\ST(A_a)$, with equality holding only if $\Aes = 2^{n-1} - 2^{k-1}$. To see this, observe that the difference of these functions is
  \[
  6\left(\Aes-2^{n-1}+2^{k-1}\right)\left(a-\Aes-2^{n-1}+2^{k-2}\right),
  \]
  which is quadratic in $\Aes$, with the coefficient of $\Aes^2$ negative, and equal to zero if $\Aes = 2^{n-1} - 2^{k-1}$ or $\Aes = a - 2^{n-1} + 2^{k-2}$.

  For the remainder of the proof, we may and shall assume that the reverse of~\eqref{eq:Aes-case1} holds, i.e., that $\Aes \ge 2^{n-1} - 2^{k-1}$ and $\Aes \ge a - 2^{n-1} + 2^{k-2}$. By our choice of $H$, this means that the smallest eigenvalue $\lambda$ of $\cG_A$ satisfies
  \[
  \lambda = |A \cap H| - |A \cap H^c|  = \Aes - \Aos = 2\Aes - a \\
  \ge 2\max\left\{ 2^{n-1} - 2^{k-1}, a - 2^{n-1} + 2^{k-2} \right\} - a,
  \]
  that is,
  \[
  \lambda \ge
  \begin{cases}
    -2^{k-2} & \text{if $a \le 2^n - 2^{k-1} - 2^{k-2}$,} \\
    a - 2^n + 2^{k-1} & \text{if $a \ge 2^n - 2^{k-1} - 2^{k-2}$.}
  \end{cases}
  \]
  As every element of $G$ has order $2$, then trivially $A = -A$ and thus it follows from~\eqref{eq:STA-ecGA} and Theorem~\ref{thm:AlCh} that (recall that $0 \not\in A$)
  \begin{equation}
    \label{eq:STA-lower-eigenvalues}
    \ST(A) \ge 2^{-n}a^3 + \begin{cases}
      -2^{k-2-n}a(2^n-a) & \text{if $a \le 2^n - 2^{k-1} - 2^{k-2}$,} \\
      2^{-n}(a - 2^n + 2^{k-1})a(2^n-a) & \text{if $a \ge 2^n - 2^{k-1} - 2^{k-2}$.}
    \end{cases}
  \end{equation}
  In order to finish the proof, we shall now show that the right-hand side of~\eqref{eq:STA-lower-eigenvalues} is greater than $\ST(A_a)$ for every $a$ with $2^n - 2^k \le a < 2^n - 2^{k-1}$.

  First, denote by $f_1(a)$ the difference between the right-hand side of~\eqref{eq:STA-lower-eigenvalues} and $\ST(A_a)$ when $2^n - 2^k \le a \le 2^n - 2^{k-1} - 2^{k-2}$. That is, let
  \[
  f_1(a) = 2^{-n} a^3 + 2^{k-2-n} a^2 - 2^{k-2}a - (3a+2^k-2^{n+1})(2^n-2^k).
  \]
  Let $a_\ell = 2^n - 2^k$ and $a_r = 2^n - 2^{k-1} - 2^{k-2}$. We need to show that $f_1(a) > 0$ for each $a$ satisfying $a_\ell \le a \le a_r$. This follows because $f_1$ is cubic in $a$, with the coefficient of $a^3$ positive, and
  \begin{itemize}
  \item
    $f_1(a_r) = 2^{2k-2} - 9 \cdot 2^{3k-n-5} \ge 2^{2k-2} - 9 \cdot 2^{2k-6} = 7 \cdot 2^{2k-6} > 0$, since $k \le n-1$.
  \item
    $f_1'(a_\ell) = 2^{k-2} (5 \cdot 2^{k-n+1} - 11) \le -6 \cdot 2^{k-2} < 0$, since $k \le n-1$.
  \item
    $f_1'(a_r) = 2^{k-4} (21 \cdot 2^{k-n} - 20) \le -19 \cdot 2^{k-5} < 0$, since $k \le n-1$.
 \end{itemize}
 Indeed, since $f'_1$ is quadratic in $a$, with the coefficient of $a^2$ positive, it has only one continuous interval where it takes negative values. Therefore $f'_1(a)$ is negative for all $a \in [a_\ell, a_r]$ and thus $f_1$ is decreasing in this interval, attaining its minimum at $a=a_r$.

 Similarly, denote by $f_2(a)$ the difference between the right-hand side of~\eqref{eq:STA-lower-eigenvalues} and $\ST(A_a)$ when $2^n - 2^{k-1} - 2^{k-2} \le a < 2^n - 2^{k-1}$. That is, let
 \begin{multline*}
   f_2(a) = 2^{-n} a^3 + 2^{-n} (a-2^n+2^{k-1})a(2^n-a) - (3a+2^k-2^{n+1})(2^n-2^k) \\
   = (2-2^{k-n-1}) a^2 + (7 \cdot 2^{k-1} - 2^{n+2}) a + 2^{2n+1} +2^{2k} - 3 \cdot 2^{k+n}.
 \end{multline*}
 Let $a_\ell = 2^n - 2^{k-1} - 2^{k-2}$ and $a_r = 2^n - 2^{k-1}$. We need to show that $f_2(a) > 0$ for each $a$ satisfying $a_\ell \le a \le a_r$. This follows because $f_2$ is quadratic in $a$, with the coefficient of $a^2$ positive, and
 \begin{itemize}
 \item 
   $f_2'(a_m) = 0$, where $a_m = \frac{2^{n+3} - 7 \cdot 2^k}{8 - 2^{k+1-n}}$.
 \item
   $f_2(a_m) = \frac{7 \cdot 2^{n+2k-3} - 2^{3k}}{2^{n+2} - 2^k} > 0$, since $k \le n-1$.
 \end{itemize}

 Finally, we characterise sets achieving equality in~\eqref{eq:Z_2n}. To this end, suppose that $\ST(A) = \ST(A_a)$. This means, in particular, that~\eqref{eq:Aes-case1} holds (as otherwise $\ST(A) > \ST(A_a)$), the two inequalities in~\eqref{eq:STA-lower} hold with equality, and the right-hand side of~\eqref{eq:STA-lower} is equal to $\ST(A_a)$. As noted above, this may happen only in the following two cases.

  \medskip
  \textsc{Case 1.} $\Aes = a - 2^{n-1}$.
  \smallskip

  In this case, $\Aos = a - \Aes = 2^{n-1}$ and thus $A \supseteq H^c$. Moreover, $\ST(\Ae) = \ST^{(n-1)}\left(A_{\Aes}^{(n-1)}\right)$ and hence we may appeal to our inductive assumption. Since
  \[
  2^{n-1} - 2^k \le \Aes = a - 2^{n-1} < 2^{n-1} - 2^{k-1},
  \]
  then there is a $K < H$ with $2^k$ elements such that $H \setminus K \subseteq \Ae$ and $\Ae \cap K$ is sum-free. But then the set $A \cap K = \Ae \cap K$ is sum-free, and $\Z_2^n \setminus K = H^c \cup (H \setminus K) \subseteq A$.

  \medskip
  \textsc{Case 2.} $\Aes = 2^{n-1} - 2^{k-1}$ and $a > 2^n - 2^{k-1} - 2^{k-2}$.
  \smallskip

  Since $\ST(\Ae) = \ST^{(n-1)}\left(A_{\Aes}^{(n-1)}\right)$, we may appeal to the inductive assumption and deduce that $\Ae = H \setminus K$ for some $K < H$ with $2^{k-1}$ elements. Let $L$ be an arbitrary subgroup satisfying $K < L < H$ and $[H : L] = 2$. Such a subgroup exists as $H / K \cong \Z_2^{n-k}$ and we have assumed that $n \ge k-1$.

  Note that $\Z_2^n / L \cong Z_2^2$ and hence there are two subgroups $H_1, H_2 < \Z_2^n$ such that $[\Z_2^n : H_i] = 2$ and $H \cap H_i = L$ for each $i \in \{1,2\}$ and $H^c \cap H_1 = H^c \setminus H_2$. Define $\Ao^1 = \Ao \cap H_1 = \Ao \setminus H_2$ and $\Ao^2 = \Ao \cap H_2 = \Ao \setminus H_1$.  We claim that $\Ao^{3-i} = H^c \cap H_i^c$ for some $i \in \{1,2\}$. Before we prove the claim, let us show that its statement contradicts the assumption that $a < 2^n - 2^{k-1}$, which in turn implies that $\ST(A) = \ST(A_a)$ cannot hold in Case 2. As $H \setminus L \subseteq H \setminus K = \Ae \subseteq A$, the claim implies that $H_i^c = (H_i^c \cap H^c) \cup (H \setminus L) \subseteq A$. But $H_i$ is a subgroup of $\Z_2^n$ of index $2$ and therefore by our choice of $H$, we have
  \[
  2^{n-1} - 2^{k-1} = a_e = |A \cap H| \le |A \cap H_i| = |A| - |A \cap H_i^c| = |A| - |H_i^c| = a - 2^{n-1},
  \]
  a contradiction. Therefore, in order to complete the proof, it suffices to prove the claim. To this end, suppose that it is not true, i.e., there are $y_1 \in (H^c \cap H_2^c) \setminus \Ao$ and $y_2 \in (H^c \cap H_1^c) \setminus \Ao$. Let $x \in H \setminus L$ be such that $x + y_1 = y_2$; such an $x$ exists as $(H^c \cap H_1^c) - (H^c \cap H_2^c) = H \setminus L$. It is easy to see that this $x$ satisfies
  \[
  |(x + \Ao) \cap \Ao| = |(x+\Ao) \cap \Ao \cap (H^c \setminus \{y_2\})| \ge 2|\Ao| - |H^c \setminus \{y_2\}| > 2|\Ao| - |H^c|.
  \]
  Since $x \in H \setminus L \subseteq H \setminus K = \Ae$, then the first inequality in~\eqref{eq:STA-lower} is strict, see~\eqref{eq:AoxAo}, contradicting our assumption that $\ST(A) = \ST(A_a)$.
\end{proof}

\section{Groups of type~\tII}

\label{sec:type-II}

In this section, we prove Theorem~\ref{thm:Z_3n} and Proposition~\ref{prop:Z3Zp}. Our proof of Theorem~\ref{thm:Z_3n} will again employ simple eigenvalue analysis. Given an abelian group $G$ and an $A \subseteq G$, we define $\dG_A$ to be the directed graph with vertex set $G$ whose arcs are all ordered pairs $(x,y)$ such that $y - x \in A$. It follows from this definition that for each $A \subseteq G$,
\[
\ST(A) = e(\dG_A[A]).
\]
A straightforward adaptation of the proof of Theorem~\ref{thm:AlCh} yields the following proposition. Here, the adjacency matrix of a directed graph $\dG$ with vertex set $V$ is the $\{0,1\}$-valued $V$-by-$V$ matrix $\big(\1[(x,y) \in \dG]\big)_{x,y \in V}$.

\begin{prop}
  \label{prop:AlCh-directed}
  Let $\dG$ be an $N$-vertex directed graph whose each vertex has both the in- and the outdegree equal to $D$. If the adjacency matrix of $\dG$ has an orthogonal basis of eigenvectors with eigenvalues satisfying $\Re(\lambda) \ge r$, then for every $U \subseteq V(\dG)$,
  \[
  e(\dG[U]) \ge \frac{D}{N} |U|^2 + \frac{r}{N} |U|(N-|U|).
  \]
\end{prop}

As it was the case with undirected Cayley graphs, the characters of $G$ form a basis of eigenvectors of $\dG_A$ for every $A \subseteq G$. Moreover, the eigenvalues of $\dG_A$ are $\sum_{a \in A} \chi(A)$, where $\chi$ ranges over all $|G|$ characters of $G$. In the case $G = \Z_3^n$, as each element of $G$ has order $3$, all characters of $G$ take values in the set $\{1, e^{\frac{2\pi i}{3}}, e^{-\frac{2\pi i}{3}}\} \subseteq \mathbb{C}$ of third roots of unity. Therefore, letting $r_A$ be the smallest real part of an eigenvalue of the adjacency matrix of $\dG_A$, we have
\begin{equation}
  \label{eq:rA}
  r_A = \min\big\{|A \cap H| - |A \setminus H|/2 \colon H < G \text{ with } [G:H] = 3\big\},
\end{equation}
where we used the fact that $\Re(e^{\pm\frac{2\pi i}{3}}) = \cos\frac{2\pi}{3} = -\frac{1}{2}$. Proposition~\ref{prop:AlCh-directed} now implies that for every $A \subseteq G$,
\begin{equation}
  \label{eq:STA-Z3n-eigenvalues}
  \ST(A) \ge 3^{-n}|A|^3 + r_A|A|\left(1-3^{-n}|A|\right),
\end{equation}
where $r_A$ is the quantity defined in~\eqref{eq:rA}.

\begin{proof}[{Proof of Theorem~\textup{\ref{thm:Z_3n}}}.]
  Let $\delta = 1/1000$, let $n$ and $t$ be as in the statement of the theorem, and denote $\Z_3^n$ by $G$. Finally, fix some $A \subseteq G$ with $a = 3^{n-1} + t$ elements. Let us first consider the case when $|A \cap H| \ge t$ for every subgroup $H < G$ of index $3$. In this case, the quantity $r_A$ defined in~\eqref{eq:rA} satisfies
  \[
  r_A \ge -\frac{a-t}{2} + t = -\frac{3^{n-1}}{2} + t
  \]
  and consequently~\eqref{eq:STA-Z3n-eigenvalues} yields
  \begin{multline*}
    \ST(A) \ge 3^{-n} a^3 + \left(t - \frac{3^{n-1}}{2}\right) a \left(1 - 3^{-n}a\right) = a\left( t + 3^{-n}a(a-t) + \frac{a}{6} - \frac{3^{n-1}}{2} \right) \\
    = a \left(t + \frac{a}{3} + \frac{a}{6} - \frac{a-t}{2}\right) = \frac{3at}{2} \ge at = 3^{n-1}t + t^2.
  \end{multline*}
  Hence, for the remainder of the proof we may assume that $|A \cap H| < t$ for some $H < G$ of index~$3$.

  Fix an arbitrary $H$ with this property and let $\varphi \colon G \to \Z_3$ be a homomorphism with $\varphi^{-1}(0) = H$. For each $i \in \{0,1,2\}$, let $A_i = A \cap \varphi^{-1}(i)$ and let $a_i = |A_i|$. As $-\varphi$ is also a homomorphism, we may assume that $a_2 \ge a_1$. Considering only the Schur triples $(x,y,z) \in A^3$ that $\varphi$ maps to $(0,2,2)$, $(2,0,2)$, $(2,2,1)$, and $(1,1,2)$, we obtain
  \begin{equation}
    \label{eq:STA-Z3n-lower}
    \begin{split}
      \ST(A) & \ge 2 \sum_{x \in A_0} |(x + A_2) \cap A_2| + \sum_{x \in A_1} |(x-A_2) \cap A_2| + \sum_{x \in A_1} |(x+A_1) \cap A_2| \\
      & \ge 2a_0(2a_2 - 3^{n-1}) +a_1 (2a_2 - 3^{n-1}) + a_1(a_1+a_2-3^{n-1}) \\
      & = (a-a_2)(2a_2 - 3^{n-1} + t) + a_0(3a_2+a_0-2a).
    \end{split}
  \end{equation}
  where we used the identity $3^{n-1} + t = a = a_0 + a_1 + a_2$. Treating $a_0$ as fixed, denote the right-hand side of~\eqref{eq:STA-Z3n-lower} by $f_{a_0}(a_2)$. Observe that the function $f_{a_0}$ is quadratic in $a_2$, with the coefficient of $a_2^2$ negative. Thus, if $2 \cdot 3^{n-2} \le a_2 \le 3^{n-1}$, then
  \[
  \ST(A) \ge \min\{f_{a_0}(2\cdot3^{n-2}), f_{a_0}(3^{n-1})\}.
  \]
  Note that
  \[
  f_{a_0}(3^{n-1}) = t(3^{n-1}+t) + a_0(3^{n-1}+a_0-2t) \ge t(3^{n-1}+t),
  \]
  as $t \le \delta 3^{n-1} \le 3^{n-1}/2$. Moreover,
  \[
  f_{a_0}(2 \cdot 3^{n-2}) = (3^{n-2} + t)^2 + a_0(a_0-2t) \ge  (3^{n-2}+t)^2 - t^2 \ge t(3^{n-1}+t),
  \]
  as $t \le \delta 3^{n-1}$ and $(\frac{1}{3}+\tau)^2 - \tau^2 \ge \tau(1+\tau)$ for each $\tau \in [0,\delta]$. Therefore, it remains to consider the case $a_1 \le a_2 < 2 \cdot 3^{n-2}$ and $a_0 < t$.

  We let $\eps = 1/30$ and define
  \[
  C = \{x \in A_2 \colon |(x+A_2) \cap A_1| \ge \eps 3^{n-1}\}.
  \]
  Since clearly $\ST(A) \ge |C| \cdot \eps 3^{n-1}$, we may further assume that
  \[
  |C| \le \frac{at}{\eps 3^{n-1}} \le \frac{(1+\delta)}{\eps} \cdot t \le \frac{(1+\delta)\delta}{\eps} \cdot 3^{n-1} \le \eps 3^{n-1}.
  \]
  In the remainder of the proof, we show that this is impossible.

  Let $B = A_2 \setminus C$. By definition, for every $x \in B$, we have $x + B \subseteq A_2 + A_2 \subseteq \varphi^{-1}(1)$ and $|(x+B) \setminus A_1| \ge |B| - \eps 3^{n-1}$. Hence, for any two $x, y \in B$, we have
  \begin{multline*}
    |(x+B) \cap (y+B)| \ge |(x + B) \cap (y+B) \cap (\varphi^{-1}(1) \setminus A_1)| \\
    \ge 2(|B| - \eps 3^{n-1}) - (3^{n-1} - a_1) \ge |B| - 3\eps 3^{n-1},
  \end{multline*}
  as $|B| + a_1 = a_1 + a_2 - |C|  \ge a - a_0 - \eps 3^{n-1} > (1-\eps) 3^{n-1}$ since $a_0 < t$. This means that every element of $B-B$ has at least $|B| - 3\eps 3^{n-1}$ representations as a difference of two elements of~$B$ and hence
  \begin{equation}
    \label{eq:Z3-B-B-upper}
    |B-B| \le \frac{|B|^2}{|B| - 3\eps3^{n-1}} < 3^{n-1},
  \end{equation}
  where the last inequality follows as $|B|$ satisfies
  \[
  \left(\frac{1}{2} - \eps\right) \cdot 3^{n-1} < \frac{a-a_0}{2} - \eps 3^{n-1} \le a_2 - |C| = |B| \le a_2 < \frac{2}{3} \cdot 3^{n-1}
  \]
  and $\frac{\beta^2}{\beta - 3\eps} < 1$ for all $\beta \in (\frac{1}{2}-\eps, \frac{2}{3})$. As $|\Stab(B-B)| \le |B-B| < 3^{n-1}$, then necessarily $|\Stab(B-B)| \le 3^{n-2}$. It now follows from Kneser's theorem that
  \begin{equation}
    \label{eq:Z3-B-B-lower}
    |B-B| \ge 2|B| - |\Stab(B-B)| \ge 2|B| - 3^{n-2}.
  \end{equation}
  But now~\eqref{eq:Z3-B-B-upper} and~\eqref{eq:Z3-B-B-lower} yield
  \[
  2|B| - 3^{n-2} \le \frac{|B|^2}{|B|-3\eps3^{n-1}},
  \]
  which is impossible as $|B| \ge (\frac{1}{2}-\eps)3^{n-1}$ and $2\beta - \frac{1}{3} > \frac{\beta^2}{\beta - 3\eps}$ if $\beta > \frac{1}{2}-\eps$.
\end{proof}

\begin{proof}[{Proof of Proposition~\textup{\ref{prop:Z3Zp}}}.]
  Fix an $a \in \{p+1, \ldots, 3p\}$, let $b = 3\lceil \frac{a}{3} \rceil$, and note that $b \ge a$ and $b-p \le 3(a-p)$. It suffices to show that there is a set $B \subseteq \Z_3 \times \Z_p$ with $b$ elements satisfying $\ST(B) \le \frac{21}{9}(b-p)^2$. One such set is $B = \Z_3 \times \{x_1, \ldots, x_{b/3}\}$, where $x_1, \ldots, x_p \in \Z_p$  are as in the statement of Theorem~\ref{thm:Zp}. Clearly,
  \[
  \ST(B) = \ST(\Z_3) \cdot \ST(\{x_1, \ldots, x_{b/3}\}) = 9 \cdot \left\lfloor \frac{b-p}{2} \right\rfloor \left\lceil \frac{b-p}{2} \right\rceil \le \frac{9}{4} (b-p)^2.\qedhere
  \]
\end{proof}

\section{A removal-type lemma of Green and Ruzsa}

\label{sec:removal-type-lemma}

In this section, we prove Proposition~\ref{prop:Green-Ruzsa-optimal}

\begin{proof}[{Proof of Proposition~\textup{\ref{prop:Green-Ruzsa-optimal}}}.]
  Suppose that $\eps > 0$ and $G$ is an abelian group of order $n$ and let $A$ be an arbitrary set of at least $(1/3+\eps)$ elements of $G$ with $\ST(A) \le \eps^2n^2/2$. Similarly as in the proof of Theorem~\ref{thm:type-I}, define
  \[
  C = \{x \in A \colon |(x-A) \cap A| \ge \eps n\}.
  \]
  As clearly $\ST(A) \ge |C| \cdot \eps n$, we have $|C| \le \eps n / 2$. Let $A' = A \setminus C$. We claim that for every $x,y \in A'$, there are at least $\eps n$ representations of $x-y$ as $a-b$ with $a,b \in A$. Indeed, for each $x,y \in A'$,
  \begin{multline*}
    |(x-A) \cap (y-A)| \ge |(x-A) \cap (y-A) \cap A^c| \\
    \ge 2(|A| - \eps n) - |A^c| = 3|A| - (1+2\eps)n \ge \eps n.
  \end{multline*}
  In particular, $\ST(A) \ge |(A'-A') \cap A| \cdot \eps n$ and therefore
  \[
  |(A' - A') \cap A'| \le |(A' - A') \cap A| \le \eps n /2.
  \]
  The set $B = A' \setminus (A' - A')$ is sum-free and
  \[
  |A \setminus B| = |C| + |(A'-A') \cap A'| \le \eps n.\qedhere
  \]
\end{proof}

\section{Concluding remarks}

In this paper, we have determined the minimum number of Schur triples in a set of $a$ elements of a finite abelian group $G$ for various $a$ and $G$. We have been able to resolve this problem completely in the cases when $G$ is a cyclic group of prime order and when $G = \Z_2^n$. We have also obtained some partial results for groups of type~\tI, that is, groups whose order is divisible by a prime $p$ satisfying $p \equiv 2 \pmod 3$. In this case, we have determined the minimum number of Schur triples for all $a$ in a short interval starting from $(\frac{1}{3} + \frac{1}{3p})|G|$, which is the largest size of a sum-free set in $G$.

We believe that solving Problem~\ref{prob:main} completely would be rather difficult. There are several reasons for it. First, determining merely the largest size of a sum-free set in a general group of type~\tIII\ requires considerable effort, see~\cite{GrRu05}. Second, the `behaviour' of the function $f_G$ defined in~\eqref{eq:fG} already becomes highly `non-uniform' when $G$ ranges over groups of type~\tII. Third, even the seemingly modest task of determining $f_G$ for groups of even order, say, would most likely entail understanding $f_G$ for general $G$; simply consider the group $\Z_2 \times G$ for some `difficult' $G$.

In view of this, it could be interesting to resolve Problem~\ref{prob:main} for particular families of $G$. One natural candidate would be the cyclic groups $\Z_{2^n}$. Here, we are tempted to guess that, similarly to the cases $G = \Z_p$ and $G = \Z_2^n$, there exists an ordering of the elements of $\Z_{2^n}$ as $x_1^n, \ldots, x_{2^n}^n$ such that for every $a$, the set $\{x_1^n, \ldots, x_a^n\}$ minimises $\ST(A)$ among all $a$-element $A \subseteq \Z_{2^n}$. It is likely that one such family of sequences $(x_i^n)$ is the one defined as follows: $x_1^0 = 0$ and $x_i^{n+1} = 2x_i^n + 1$ and $x_{2^n+i}^{n+1} = 2x_i^n$ for all $n\ge 0$ and $i \in [2^n]$.

\medskip
\noindent
\textbf{Acknowledgement.} Parts of this work were carried out when the first author visited the Institute for Mathematical Research (FIM) of ETH Z\"urich, and also when the second author visited the School of Mathematical Sciences of Tel Aviv University. We would like to thank both institutions for their hospitality and for creating a stimulating research environment. We are indebted to B\'ela Bajnok for pointing out an error in the statement of Theorem~\ref{thm:Zp} in the previous version of this paper.

\bibliographystyle{amsplain}
\bibliography{additive-triples}

\end{document}